\documentclass[11pt,a4paper]{article}
\usepackage{a4,amssymb,amsmath,amsthm,url,xcolor,enumerate,float}
\usepackage{bezier,amsfonts,amssymb,graphicx,amsthm,url,lineno}
\usepackage[english]{babel}
\title{The Saturation Number for the length of Degree Monotone Paths}
\date{}
\begin{document}
\newtheorem{theorem}{Theorem}[section]
\newtheorem{definition}{Definition}[section]
\newtheorem{proposition}[theorem]{Proposition}
\newtheorem{corollary}[theorem]{Corollary}
\newtheorem{lemma}[theorem]{Lemma}
\newcommand*\cartprod{\mbox{ } \Box \mbox{ }}
\newtheoremstyle{break}
  {}
  {}
  {\itshape}
  {}
  {\bfseries}
  {.}
  {\newline}
  {}

\theoremstyle{break}

\newtheorem{propskip}[theorem]{Proposition}
\DeclareGraphicsExtensions{.pdf,.png,.jpg}
\author{Yair Caro \\ Department of Mathematics\\ University of Haifa-Oranim \\ Israel \and Josef  Lauri\\ Department of Mathematics \\ University of Malta
\\ Malta \and Christina Zarb \\Department of Mathematics \\University of Malta \\Malta }

\maketitle{}
\begin{abstract}
A degree monotone path in a graph $G$ is a path $P$ such that the sequence of degrees of the vertices in the order in which they appear on $P$ is monotonic.  The length of the longest degree monotone path in $G$ is denoted by $mp(G)$.  This parameter, inspired by the well-known Erdos-Szekeres theorem, has been studied by the authors in two earlier papers.  Here we consider a saturation problem for the parameter  $mp(G)$.  We call $G$ saturated if, for every edge $e$ added to $G$, $mp(G+e) >mp(G)$, and we define $h(n,k)$ to be the least possible number of edges in a saturated graph $G$ on $n$ vertices with $mp(G) < k$, while $mp(G+e) \geq k$ for every new edge $e$.

We obtain linear lower and upper bounds for $h(n,k)$, we determine exactly the values of $h(n,k)$ for $k=3$ and $4$, and we present constructions of saturated graphs.
\end{abstract}

\section{Introduction}

Given a graph $G$, a degree monotone path is a path $v_1v_2 \ldots v_k$ such that $deg(v_1) \leq deg(v_2) \leq \ldots \leq deg(v_k)$ or $deg(v_1) \geq deg(v_2 \geq \ldots \geq deg(v_k)$.  This notion, inspired by the well-known Erd{\H{o}}s-Szekeres Theorem \cite{eliavs2013,erdos1935combinatorial}, was introduced in \cite{updownhilldom} under the name of uphill and downhill path in relation to domination problems, also studied in \cite{deering2013uphill,downhill2,downhill}.  In \cite{updownhilldom}, the study of the parameter $mp(G)$, which denotes the length of the longest degree monotone path in $G$, was specifically suggested.  This parameter was  studied by the authors in \cite{dmpclz,dmp2}, and among many results obtained, the parameter $f(n,k) =\max\{ |E(G)|: |V(G)| = n, mp(G) < k\}$ was also defined.  It was shown that this is closely related to the Turan numbers $t(n,k)  = \max \{|E(G)|: |V(G)| = n, G \mbox{ contains no copy of } K_k\}$.

A general form of the Turan numbers with respect to a graph $H$ is $t(n,H)=\max \{|E(G)|: |V(G)| = n, G \mbox{ contains no copy of } H\}$.  The study of Turan numbers is undoubtedly  considered as one of the fundamental problems in extremal graph and hypergraph theory \cite{bollobas2004extremal}.

The Turan number has a counter-part known as the saturation number with respect to a given graph $H$, defined as 

\begin{align*}
sat(n,H)= \min\{|E(G)|: |V(G)| = n, G \mbox{ contains no copy of } H,\\
  \mbox{ but } G+e \mbox{ contains } H \mbox{ for any edge added to } G \}
\end{align*}

Tuza and Kaszonyi in \cite{kaszonyi1986saturated} launched a systematic study of $sat(n,H)$ following an earlier result by Erdos, Hajnal  and Moon \cite{erdos1964problem} who proved that $sat(n,K_k)= \binom{k-2}{2}+(k-2)(n-k+2)$ with a unique graph attaining this bound, namely $K_{k-2}+\overline{K}_{n-k+2}$.   For the current paper, it is worth noting that $sat(n,P_k)$ ($sat(n,k)$ for short) is known \cite{kaszonyi1986saturated} for every $k$ for $n$ sufficiently large with respect to $k$, and in particular for $n$ large enough, $sat(n,k)=n(1-c(k))$  .  For a survey and recent information about saturation , see \cite{faudree2011}.

In this spirit, we call a graph $G$ saturated if $mp(G+e)>mp(G)$, for all new edges $e$ joining non-adjacent vertices in $G$.  If it happens that $mp(G+e) \geq k$ for all new edges $e$ we sometimes refer to the saturated graph $G$ as $k$-saturated.  By convention we say that $K_m$ is $k$-saturated for $m \leq k-1$.  Then we define \[h(n,k)=\min\{|E(G)|:|V(G)|=n, \mbox{$G$ is $k$-saturated}\}.\]  

In Section 2, we prove linear lower and upper bounds for this parameter.  In Section 3, we provide  exact  determination of $h(n,k)$ for $k=3,4$.  In Section 4 we present several open problems concerning $h(n,k)$ for $k \geq 5$ as well as several other problems and conjectures.

\section{General Lower and Upper bounds}

\subsection{Lower bounds}

We begin by showing that $sat(n,k)$ is a lower bound for $h(n,k)$.

\begin{proposition} \label{prop1}
\mbox{}\\

For $k \geq 2$, $h(n,k) \geq sat(n,k)$.
\end{proposition}

\begin{proof}
Clearly, if $G$ is a graph realising $sat(n,P_k)=sat(n,k)$,  this means that $G$ does not contain a copy of $P_k$, and hence no degree monotone path of length $k$.  But $G+e$ contains $P_k$, but not necessarily a degree monotone path of length $k$.  Hence $h(n,k) \geq sat(n,k)$.
\end{proof}

Recall  that for fixed $k$ and large $n$,  $sat(n,k)  = n( 1- c(k))  < n$.  We now strengthen Proposition \ref{prop1} to show that for $k \geq 4$, $h(n,k) \geq n$.   We first prove a lemma, and subsequently  a corollary, which will then be used in the main proof.

\begin{lemma} \label{lemma_1}
Let $G$ be a connected graph with a vertex $u$ of degree 1 and a vertex $v$ of maximum degree $\Delta \geq 2$ which are not adjacent.  Then $mp(G+uv) \leq mp(G)$, namely $G$ is not saturated.
 \end{lemma}
\begin{proof}
Let $H=G+uv$ and let $P$ be a path in $H$ which realizes $mp(H)$.  Let $u^*$ and $v^*$ be the vertices $u$ and $v$ as they appear in $H$.  

If  $\Delta=2$,  then clearly $G$  is a path on $k\geq 4$ vertices and $mp(G)=k-1$ , and if we take $u$ to be the first vertex of the path, and $v$ to be the $(k-1)^{th}$ vertex, then $mp(H)=k-1= mp(G)$.

So we may assume $\Delta \geq 3$.  Now if $u^*$ and $v^*$ are not on $P$, then $P$ is degree monotone in $G$ and hence $mp(H) \leq mp(G)$.  If $v^*$ is on $P$ but $u^*$ is not then $v^*$ must be the last vertex on $P$, and hence, in $G$, the path $P$ with $v^*$ replaced by $v$ is also degree monotone in $G$ and  $mp(H) \leq mp(G)$.  Similarly, if $u^*$ is on $P$ but $v^*$ is not, then $u^*$ must be the first vertex on $P$, since clearly $u^*$ cannot be in the ``middle" of the path as then the next vertex on $P$ must be $v^*$, which is not on $P$.  Then the path $P$ in $G$ with $u^*$ replaced by $u$ is also degree monotone in $G$ and again  $mp(H) \leq mp(G)$.  If $u^*$  is the last vertex on the path then clearly $P$ is not maximal, as $P \cup \{v^*\}$ via the edge $u^*v^*$  is a longer degree monotone path, contradicting maximality of $P$.

So the only remaining case to consider is when  $u^*$ and $v^*$ are both on $P$.  Then clearly $v^*$ is the last vertex on $P$.  If $u^*$ is the first vertex, then either $P=u^*v^*$ and $mp(H)=2 \leq mp(G)$, or the path $P$ is degree monotone in $G$ too.  If $u^*$ is not the first vertex,  then the next vertex on $P$ must be $v^*$ which is the last vertex.  Hence, in this case, all predecessors of $u^*$ on $P$ must have degree at most 2.  Buf if the first vertex $y$ in $P$ has degree 1, then, in $G$, the path $y \ldots u$ is disconnected from the rest of $G$, which is impossible.  Therefore $deg(y)=2$ and $y$ has a  neighbour $w$ which must have degree greater than 2 (note that $w$ may be equal to $v^*$ but cannot be any other vertex on $P$).  But then, the path $u \ldots yw$ is degree monotone in $G$ and is of the same length as $P$, and hence $mp(H) \leq mp(G)$.
\end{proof}

Lemma \ref{lemma_1} is best possible with respect to the adjacency condition between minimum degrees and maximum degrees because, if the minimum degree is greater than 1, and a vertex $u$ of minimum degree is not adjacent to vertex $v$, then $mp(G+uv)$ may be larger than $mp(G)$.  As an example, consider graph $G_n$ made up of the cycle $C_{2n}$, $n \geq 3$, with vertices labelled $v_1,v_2, \ldots,v_{2n}$, and a vertex $w$ connected to vertices $v_1,v_3,v_5,\ldots,v_{2n-1}$.  Thus $w$ has degree $\Delta=n$ and $\delta=2$, and $mp(G_n)=3$.  The vertices of degree 2 are not connected to $w$, but connecting any such vertex to $w$ by an edge $e$ gives $mp(G_n+e)=5$.  In fact, these graphs are 5-saturated even though they have non-adjacent vertices of maximum degree $\Delta \geq 3$ and minimum degree $\delta=2$.

\begin{corollary} \label{cor_1}
Let $T$ be a tree on $n \geq 3 $ vertices.  Then $T$ is saturated for a degree monotone path if and only if $T=K_{1,n-1}$.
\end{corollary}

\begin{proof}
Suppose first $mp(T)\geq3$.  Then clearly $T$ is a not a star, hence there is a leaf not connected to  a vertex of maximum degree and by Lemma  \ref{lemma_1}, $T$  is not saturated.
   
So suppose $mp(T) = 2$. If not all leaves are adjacent to the same vertex of maximum degree then again by Lemma \ref{lemma_1}, $T$ is not saturated. Hence $T$  must be a star $K_{1,n-1}$.
 
Indeed $ K_{1,n-1}$  is saturated and $mp(K_{1,n-1}) =  2$ while $mp(K_{1,n-1} +e)  = 3$ for every edge $ e \not \in E(K_{1,n-1})$.
\end{proof}

\begin{theorem} \label{lower}
For $n \geq 3$ and $k \geq 4$, $h(n,k) \geq n.$
\end{theorem}

\begin{proof}

We may assume that $n \geq k$ for otherwise, trivially, $K_n$ is saturated having $ \binom n  2 \geq n$ edges for $n \geq 3$

So let $G$ be a graph on $n \geq k$ vertices realizing $h(n,k)$, $k \geq 4$.  If $G$ is connected then by  Corollary \ref{cor_1}, $G$ is not a tree hence $|E(G)| \geq n$ as required.  

So we may assume that $G$ is not connected, and let $G_1,G_2, \ldots G_t$ be the connected components of $G$.  Again, by Corollary \ref{cor_1}, we infer that every component on at least three vertices is not a tree and hence must have at least $|V(G_j)|$ edges.

If there are two components $G_i$ and $G_j$ on at most two vertices, adding an edge joining these two components does not create a degree monotone path of length 4 or more, contradicting the fact that $G$ is saturated.

If there is just one component on at most two vertices, then one can connect one vertex of this component to a vertex of maximum degree in another component, and again no degree monotone path of length four or more is created, contradicting the fact that $G$ is saturated.

Hence \[|E(G)|= \sum_{i=1}^{i=t} |E(G_i)| \geq \sum_{i=1}^{i=t} |V(G_i)|=n,\]
and therefore $h(n,k) \geq n$ for $n \geq 3$ and $k \geq 4$.
\end{proof}

\subsection{Upper bounds}

We now give a linear upper  bound for $h(n,k)$.   We consider separately $k$ odd and $k$ even.

We first recall the definition of the Cartesian product $G \cartprod H$ for two graphs $G$ and $H$.  The vertex set of the product is $V(G) \times V(H)$.  Two vertices $(u_1,v_1)$ and $(u_2,v_2)$ are adjacent if eaither $u_1$ and $u_2$ are adjacent in $G$ and $v_1=v_2$, or $v_1,v_2$ are adjacent in $H$ and $u_1=u_2$.

\begin{theorem} \label{thmcartprod}
If $k\geq 3$ is an odd integer, then  $h(n,k) \leq  \frac{n(3k-1)}{12}$  for  $n =  0 \pmod {\frac{3(k-1)}{2}}$.
\end{theorem}

\begin{proof}
Consider the graph $G=P_3 \cartprod K_{t}$, for $k\geq 3$ odd and $t=\frac{k-1}{2}$.  Clearly $|V(G)|=\frac{3(k-1)}{2}$ and $|E(G)|=\frac{3(k-1)(k-3)}{8} + \frac{2(k-1)}{2}=\frac{(k-1)(3k-1)}{8}$.  For $k=3$ (so $t=1$)  this is simply $P_3$ and $mp(P_3)=2$ while for $k=5$ (so $t=2$), this gives the graph $G = P_3 \cartprod K_2$, which is $C_6$ plue one edge joining two antipodal vertices and clearly $mp(G)=4$. 

We now show that this graph, which has $mp(G)=k-1$, is saturated. In $G=P_3 \cartprod K_{t}$, let the top $t$ vertices be $u_1, \ldots, u_t$,  all having degree $t$, the middle vertices $v_1,\ldots,v_t$ all having degree $t+1$, and the bottom vertices $w_1,\ldots, w_t$ all having degree $t$.  It is clear that $mpG)=2t=k-1$, taking for example the path $u_1 \ldots u_t v_t \ldots v_1$. Because of the symmetry of $G$, we only need to check the addition of the edges $u_1v_2$, $v_1w_2$ and $u_1w_1$.  

\begin{itemize}
\item{If the edge $u_1v_2$ is added, then the path $w_1 \ldots w_t v_t \ldots v_3 v_1u _1v_2$ has exactly $t+t-2+3=2t+1=k$ vertices.}
\item{If the edge $v_1w_2$ is added, then the path $u_1 \ldots u_t v_t \ldots v_2 w_2 v_1$ has exactly $t+t-1+2=2t+1=k$ vertices.}
\item{If the edge $u_1w_1$ is added, then the path $u_2 \ldots u_t v_t \ldots v_1 w_1 u_1$ has exactly $t-1+t+2=2t+1=k$ vertices.}
\end{itemize}

Hence $G$ is saturated with $mp(G)=k-1$.

We now consider two disjoint copies of $G$, $G_1$ and $G_2$.  We label this graph $2G$ and show that this graph is also saturated.  Again labelling the vertices of $G$ as above, by the symmetry of $G$ we only need to consider the addition of the edges joining $u_t$ in $G_1$ to $u_1$ in $G_2$, $u_t$ in $G_1$ to $v_1$ in $G_2$, and  $v_t$ in $G_1$ to $v_1$ in $G_2$:

\begin{itemize}
\item{If the edge  joining $u_t$ in $G_1$ to $u_1$ in $G_2$ is added, then the path $u_1 \ldots u_t$ in $G_1$ followed by $u_1 v_1 \ldots v_t$ in $G_2$ has exactly $t+t+1=2t+1=k$ vertices.}
\item{If the edge  joining $u_t$ in $G_1$ to $v_1$ in $G_2$ is added, then the path $v_1 \ldots v_tu_t$ in $G_1$  followed by $v_1  \ldots v_t$ in $G_2$ has exactly $t+1+t=2t+1=k$ vertices.}
\item{If the edge joining $v_t$ in $G_1$ to $v_1$ in $G_2$, is added, then the path $u_t \ldots u_1 v_1  \ldots v_t$ in $G_1$ followed by $v_1$ in $G_2$  has exactly $2t+1=k$ vertices.}

\end{itemize}

Hence $2G$ is saturated, and clearly  this also applies to $p \geq 3$ disjoint copies of $G$, $pG$.  Now $pG$ has $n=p\frac{3(k-1)}{2}$ vertices and $p\frac{(k-1)(3k-1)}{8}$ edges.  Hence, for  $n =  0 \pmod {\frac{3(k-1)}{2}}$, the number of edges is $\frac{n(3k-1)}{12}$, as stated.

\end{proof}

\begin{lemma} \label{lemma_even}
Let $G$ be a saturated graph with $mp(G)=k$.  Consider the graph $H=G+v$, where $v$ is a new vertex connected to all the vertices of $G$.  Then $mp(H)=k+1$, and $H$ is saturated.
\end{lemma}

\begin{proof}
Consider the graph $H$.  Then $deg(v)=|V(G)|$ and $v$ has maximum degree.  So any degree monotone path in $G$ can be extended in $H$ by including vertex $v$, and hence $mp(H)=mp(G)+1=k+1$.

Now since $G$ is saturated, adding any edge $e$ creates a degree monotone path of length $k+1$, and hence, adding the same edge in $H$ creates a path of length $k+2$.  The only edges which can be added in $H$ are those that can be added in $G$, and hence $H$ is saturated with $mp(H)=k+1$ as required.
\end{proof}

This lemma, together with Theorem \ref{thmcartprod}, leads to the following result:

\begin{theorem} \label{keven}
For $k \geq 4$ and $k=0 \pmod 2$, $h(n,k) \leq \frac{n(3k+8)(k-2)}{4(3k-4)}$ for $n=0 \pmod {\frac{3k-4}{2}}$.
\end{theorem}

\begin{proof}

In Theorem \ref{thmcartprod}, we proved that $G=P_3 \cartprod K_t$, where $t=\frac{j-1}{2}$ has $mp(G)=j-1$ and is saturated for $j \geq 3$ and $j$ odd.  Now by Lemma \ref{lemma_even}, $H=G+v$ has $mp(H)=j+1$ (even) and is saturated.  Then $H$ has $\frac{3(j-1)}{2}+1 = \frac{3j-1}{2}$ vertices and $\frac{(j-1)(3j-1)}{8}+\frac{3(j-1)}{2}=\frac{(3j+11)(j-1)}{8}$ edges. Now let $k=j+1$, and hence we have $\frac{3k-4}{2}$ vertices and $\frac{(3k+8)(k-2)}{2}$ edges. 

We now consider two disjoint copies of $H$, $H_1$ and $H_2$ and call this graph $2H$.  We need only consider edges which involve the new vertex of degree $\frac{3(k-2)}{2}$, which has the largest degree, as other edges have the same effect as they have in $2G$.  If we connect the vertex of degree $\frac{3(k-2)}{2}$ in $H_1$ to that of the same degree in $H_2$, we can take a path of length $k-1$ in $H_1$ ending with the vertex of maximum degree and then move to the vertex in $H_2$, giving a path of length $k$.  If we connect the vertex of degree $\frac{3(k-2)}{2}$ in $H_1$  to one of degree $\frac{k}{2}$ in $H_2$, then we take a path of length $k-1$ in $H_2$ ending with the vertex connected to the vertex in $H_1$, and then move to this vertex in $H_1$ to give a degree monotone path of length $k$.  Finally, if we connect the vertex of degree $\frac{3(k-2)}{2}$ in $H_1$  to one of degree $\frac{k+2}{2}$ in $H_2$, then we can take a degree  monotone path in $H_2$ of length $k-1$ ending on the vertex connected to $H_2$, and then the vertex in $H_2$ to give a degree monotone path of length $k$ in $2H$.

Hence $2H$ is saturated and this also applies to $p \geq 3$ disjoint copies of $H$, $pH$.  This graph has $n=p\frac{3k-4}{2}$ vertices and $p\frac{(3k+8)(k-2)}{8}$ edges.  Hence for $n=0 \pmod {\frac{3k-4}{2}}$, the number of edges is $\frac{n(3k+8)(k-2)}{4(3k-4)}$ as stated.

\end{proof}

We next show, as an example, how to extend the results given in Theorems  \ref{thmcartprod} and \ref{keven} , to the case where $n \not = 0 \pmod {f(k)}$, where $f(k)$ is the modulus given in these theorems.  We will demonstrate it in the case $k=5$.    

\begin{proposition} \label{mod6}
For $n \geq 8$, $h(n,5) \leq \frac{7n+c(n \pmod 6)}{6}$, where $c(n \pmod 6)=\{0,35,16,27,8,28\}$ for $n \pmod 6=0,1,2,3,4,5$ respectively.
\end{proposition}

\begin{proof}
Consider the graphs $G=P_3 \cartprod K_2$, $H=K_5-e$ for $e \in E(K_5)$ and $K_4$, which are sturated for $k=5$ and clearly  $mp(G)=mp(H)=mp(K_4)=4$.  Every integer $n \geq 8$ can be represented in the form $6x+5y+4z$ with $x,y,z$ non-negative integers.  Hence $x$ copies of $G$, $y$ copies of $H$ and $z$ copies of $K_4$ produce graphs for every $n \geq 8$.  It is easy to check that any graph made up of two vertex disjoint copies of any combination of $G$, $H$ and $K_4$  is also saturated, and hence any combination of  vertex disjoint copies of these graphs is saturated.

Hence any graph made up of a disjoint combination of  any number of these three graphs is saturated.

For $n= 0 \pmod 6$, the result follows immediately by substituting for $k=5$ in Theorem \ref{thmcartprod}.

For  $n= 1 \pmod 6$, we take the graph made up of $\frac{n-13}{6}$ copies of $G$, two copies $K_4$ and one copy of $H$.  The graph thus obtained is saturated and has $\frac{7(n-13)}{6}+ 12+ 9 = \frac{7n+35}{6}$ edges.

For $n= 2 \pmod 6$, we take the graph made up of $\frac{n-8}{6}$ copies of $G$ and  two copies $K_4$.  The graph thus obtained is saturated and has $\frac{7(n-8)}{6}+ 12 = \frac{7n+16}{6}$ edges.

For  $n= 3 \pmod 6$, we take the graph made up of $\frac{n-9}{6}$ copies of $G$, one copy of $K_4$ and one copy of $H$.  The graph thus obtained is saturated and has $\frac{7(n-9)}{6}+ 6+ 9 = \frac{7n+27}{6}$ edges.

For  $n= 4 \pmod 6$, we take the graph made up of $\frac{n-4}{6}$ copies of $G$ and one copy of  $K_4$ .  The graph thus obtained is saturated and has $\frac{7(n-4)}{6}+ 6 = \frac{7n+8}{6}$ edges.

For  $n= 5 \pmod 6$, we take the graph made up of $\frac{n-5}{6}$ copies of $G$ and one copy of $H$.  The graph thus obtained is saturated and has $\frac{7(n-5)}{6}+ 9 = \frac{7n+28}{6}$ edges.

\end{proof}

Note : Applying the technique demonstrated in Proposition \ref{mod6},  we can extend Theorems \ref{thmcartprod} and \ref{keven} to cover all $n \geq (k-1)(k-2)$, and we state it rather crudely as follows :  
 
\begin{enumerate}
\item{For $k \geq 3$, $k = 1 \pmod 2$  and $n \geq (k-1)(k-2)$, $h(n,k) \leq  \frac{n(3k-1)}{12} + O(k^2)$.}
\item{For $k \geq 4$, $k = 0 \pmod 2$  and $n \geq  (k-1)( k-2)$, $h(n,k)  \leq \frac{n(3k+8)(k-2)}{4(3k-4)} + O(k^2)$.}
\end{enumerate}

\section{Determination of $h(n,k)$ for $k=2,3,4$.}

We first determine the exact value of $h(n,2)$ and $h(n,3)$.

\begin{proposition}
\mbox{}
\begin{enumerate}
\item{$h(n,2)=0$.}
\item{$h(n,3)=\frac{n}{2}$ for $n=0 \pmod 2$, while $h(n,3)= \frac{n+1}{2}$ for $n=1 \pmod 2$.}
\end{enumerate}
\end{proposition}

\begin{proof}
\mbox{}\\
\noindent 1. \indent $mp(G)=1$ if and only if $G$ is a graph with no edges, and any edge we add gives $mp(G+e)=2$.
\medskip

\noindent 2. \indent By Proposition \ref{prop1}, $h(n,3) \geq sat(n,3)=\lfloor \frac{n}{2} \rfloor$.  Consider $n=0 \pmod 2$.  Let $G$ be made up of $\frac{n}{2}$ copies of $K_2$.  This is the only graph which achieves $sat(n,3)$.  Clearly $mp(G)=2$, and adding any edge will create a copy of $P_4$ so $mp(G+e)=3$.

Now if $n=1 \pmod 2$, the graph $G$ made up of $\lfloor \frac{n}{2} \rfloor$ copies of $K_2$ and one copy of $K_1$ achieves $sat(n,3)$, and is the only such graph.  Again $mp(G)=2$.  If we add an edge joining two vertices from disjoint copies of $K_2$ then we get a copy of $P_4$ and $mp(G+e)=3$;  however, if we add a vertex joining a vertex from $K_2$ to the vertex in $K_1$, this gives a copy of $P_3$, and $mp(G+e)=2$, hence $h(n,3)  \geq sat(n,3)+1$.

Consider the graph $G$ made up of $\frac{n-3}{2}$ copies of $K_2$, and a single copy of $P_3$. Again it is clear that $mp(G)=2$.  Adding an edge  joining two vertices from disjoint copies of $K_2$ then we get a copy of $P_4$ and $mp(G+e)=3$, while adding an edge joining a vertex from $K_2$ to one in $P_3$ gives $mp(G+e)=4$.  The number of edges in this graph is $\frac{n+1}{2}= sat(n,3)+1$ as stated.
\end{proof}

We now determine the exact value of $h(n,4)$.  For this we need another lemma:

\begin{lemma} \label{lemma_2}
Let $G$ be a saturated connected graph with $|E(G)| \leq |V(G)|$ and $2 \leq mp(G) \leq 3$. Then 
\begin{enumerate}
\item{If $mp(G)=2$ then $G=K_{1,\Delta}$ and for $\Delta \geq 2$, $mp(G+e)=3, \forall e \not \in E(G)$.}
\item{If $mp(G)=3$ then $G=K_3$ which is saturated by definition.}
\end{enumerate}
\end{lemma}
\begin{proof}
Let $G$ be such a graph.  Then since $|E(G)| \leq |V(G)|$, $G$ is either a tree or is unicyclic.

If $G$ is a tree then either all leaves are adjacent to the same vertex which has maximum degree, that is $G=K_{1,\Delta}$.  Then $mp(G)=2$ and, in case $\Delta \geq 2$,  adding any edge between two leaves $u$ and $v$ gives $mp(G+uv)=3$.  If $G$ is a tree but not $K_{1,\Delta}$, then there is a leaf $u$ and a vertex $v$ of maximum degree which are not adjacent, and hence by Lemma \ref{lemma_1}, $G$ is not saturated.

So suppose $G$ is unicyclic.  Then it cannot be a simple cycle $C_n$ on $n \geq 4$ vertices as otherwise $mp(C_n)=n \geq 4$. Observe that $C_3=K_3$ is saturated by definition.

So $G$ is unicyclic with at least one leaf if the cycle has at least four vertices. 

Suppose $mp(G)=2$.   If there are at least two vertices on the cycle which have branches attached, then on one of these branches (including the vertex on the cycle) there must be a vertex of maximum degree, and on the other branch there must be a leaf not connected to this vertex of maximum degree, and hence by Lemma \ref{lemma_1} $G$ is not saturated.

So there is precisely one vertex on the cycle with degree greater than two, which means that $mp(G)>2$, a contradiction.

So now suppose $mp(G)=3$. If there are at least two vertices on the cycle which have branches attached, then on one of these branches (including the vertex on the cycle) there must be a vertex of maximum degree, and on the other branch there must be a leaf not connected to this vertex of maximum degree, and hence by Lemma \ref{lemma_1} $G$ is not saturated.

So there is precisely one vertex on the cycle with degree greater than two, and if  the cycle has at least four vertices, $mp(G)\geq 4$, a contradiction.

So it remains to consider the cycle $K_3$ with exactly one vertex $x$ with degree greater than two.  

Suppose the vertex $x$ has $p$  leaves and $q$ branches  with $p,q \geq 0$.  We consider several cases :  
 
\begin{enumerate}
\item{ If $p \geq 2$, we connect  two leaves to get $H$ with $mp(H) = mp(G) = 3$, and $G$ is not saturated.  Hence $p \leq 1$.}
 \item{I f $p = 1$ and $q \geq 1$, then  either $x$ is of maximum degree $\Delta \geq 3$,  and there is a leaf  not connected to $x$, so by Lemma \ref{lemma_1} $G$ is not saturated, or  there is a vertex of maximum degree in one of these branch, so the leaf at $x$ is not connected to the vertex of maximum degree and again by Lemma \ref{lemma_1}, $G$ is not saturated.}
\item{If $p = 1$ and $ q = 0$, then $G$ is $K_3$  with a leaf attached and clearly it  is not saturated.}
\item{If  $p = 0$ and   $q \geq 2$, then  either $x$ is of maximum degree $\Delta \geq  3$ and there is a leaf in the branch not connected to $x$, so by Lemma \ref{lemma_1} $G$ is not saturated, or  there is a vertex of maximum degree in one of these branches, so the leaf at $x$ is not connected to the vertex of maximum degree and again by Lemma \ref{lemma_1}, $G$ is not saturated.}
\item{If $p = 0$ and $q = 1$, then $deg(x) = 3$.  Let $z$  be the neighbour of $x$ in this branch.  If $deg(z) \geq 3$ then $mp(G) \geq 4$, a contradiction. Hence $deg(z) = 2$, and let $w$ be the neighbour of $z$
 
If $deg(w) =1$ then $x$ has maximum degree, $w$ is not connected to $x$ and by Lemma \ref{lemma_1}, $G$  is not saturated.  So $deg(w)  \geq 2$.  we consider two cases:

\noindent \emph{Case 1: $deg(w)= 2$}.

Let $u$ be the neighbour of $w$.  If $deg(u)  \leq 2$, then we have $uwzv$ a degree monotone path of length four. So $deg(u) \geq 3$.
 
If $ deg(u)>3$ then if the edge $xw$ is  added, $mp(G+xw) = 3$ and $G$  is not saturated.  Hence $deg(u)=3$.  Let $s$ and  $y$ be the neighbours of $u$.  If either $s$ or $y$ have degree at least three,  we have $zwux$ or $zwuy$ degree montone paths of length four, a contradiction. So both $s$ and $y$ have degree at most two.
 
If either $s$ or $y$ is a leaf,  say $s$, then either $\Delta=3$ and $s$ is  leaf is not connected to $x$, so by Lemma \ref{lemma_1}  $G$ is not saturated, or  $\Delta \geq 4$ and is realized by a vertex $r$ say  on the branch at $y$.  Again $s$ is a leaf not  adjacent to $r$ and by lemma 2,2  $G$ is not saturated .
 
So $deg(s)= deg(y) = 2$, and either the maximum degree $\Delta=3$, and there is a leaf not adjacent to $x$, so by Lemma \ref{lemma_1} $G$ is not saturated, or there is a vertex $r$ of maximum degree $\Delta \geq 4$ which is on one of the branches starting at $s$ or $y$, say $s$.  But then there is a leaf on the branch starting at $y$ not adjacent to the vertex $r$, and again by Lemma \ref{lemma_1} $G$ is saturated.

\noindent \emph{Case 2: $deg(w)= t \geq 3$}.

Let $x_1, \ldots ,x_ t$ be the neighbors of $w$.  If for some $j$, $deg(x_j) = 1$,  then either $\Delta=3$ and $x_j$ is not connected to $x$, so by Lemma \ref{lemma_1} $G$ is not saturated, or $\Delta \geq 4$  and is realized by a vertex $r$ on a branch at some $x_i$, $i \not = j$ . Then $x_j$ is a leaf not adjacent to $r$ and by Lemma \ref{lemma_1} $G$ is not saturated.
 
So $deg(x_j) \geq  2$ for $j = 1, \ldots,t$.    Now  if $\Delta=,3$ then a leaf on one these branches starting at $x_1, \ldots ,x_ t$   is not connected to $x$ and by Lemma \ref{lemma_1} $G$ is not saturated.  Otherwise $\Delta \geq 4$ and  a vertex $r$ of maximum degree appears on the branch starting at say $x_j$ . Then a leaf on  any other branch is not connected to $r$ and by Lemma \ref{lemma_1} $G$ is not saturated.}
\end{enumerate}
Hence $G=K_3$ is the only saturated graph with $|E(G)| \leq |V(G)|$ and $mp(G)=3$.

\end{proof}

\begin{theorem}
For $n = 0 \pmod 3$, $h(n,4)=n$ while for $n=1,2 \pmod 3$, $h(n,4)=n+1$.
\end{theorem}

\begin{proof}

We first prove the upperbound for $h(n,4)$.  Consider the following cases:
\begin{enumerate}
\item{Assume  $n = 0 \pmod 3$. Let $G$ be made up of $\frac{n}{3}$ copies of $K_3$, then clearly $mp(G) = 3$. Any edge we add  gives a degree monotone path of length 4. So $G$ is saturated and hence $h(n,4) \leq n$, for $n = 0 \pmod 3$.}
\item{Assume $n = 1\pmod 3$ .  Let $G$ be made up of $\frac{n-4}{3}$ copies of $K_3$  and a copy of $K_4 - e$, $e \in E(K_4)$.  Clearly $mp(G) = 3$ and it is easy to see that $mp(G+e)\geq 4$.  So $G$ is saturated and hence $h(n,4) \leq n+1$, for $n =1 \pmod 3$.}
\item{Assume $n = 2\pmod 3$ .  Let $G$ be made up of $\frac{n-5}{3}$ copies of $K_3$  and two copies of $K_3$ with a common vertex.  Clearly $mp(G) = 3$ and it is easy to see that $mp(G+e)\geq 4$.  So $G$ is saturated and hence $h(n,4) \leq n+1$, for $n =2 \pmod 3$.}
\end{enumerate}

Now to the lower bound.  Suppose $G$ is a graph on $n \geq 3$ vertices realising $h(n,4)$.  If $G$ is connected than by Lemma \ref{lemma_2}, either $G$ is $K_3$ or $|E(G)| \geq n+1$.  Hence we may assume that $G$ is not connected, and let $G_1,G_2, \ldots, G_t$ be the connected components of $G$.

Again, by Lemma \ref{lemma_2}, every component $G_ j$ on at least 3 vertices is either $K_3$ or contains at least $|V(G_ j)| +1$ edges.
 
If there are at least two components say $G_ i$ and $G_ j$ on at most two vertices each, then we can just add an edge between a vertex in $G_i$ and one in $G_j$ without creating a degree monotone path of length more than 3,  contradicting the fact that $G$ is saturated.
 
Lastly if there is just one component $G_ j$ on at most two vertices, then if we connect a vertex in this component to a vertex $v$ of maximum degree in another component of $G$ , then clearly no degree monotone path of length 4 or more is created , once again contradicting that $G$ is saturated.
 
Hence all components of $G$ have at least 3 vertices.  If there are at least two components which are not $K_3$ then $|E(G)|  \geq n +2$, and this is not optimal by the constructions above.  If there is just one component which is not $K_3$, then $|E(G)| \geq n+1$ and so for $n = 1 ,2 \pmod 3$, $h(n,4) \geq n+1$ proving the constructions above are optimal .
 
Finally, if all components are $K_3$, then $|E(G)|=n$, proving $h(n,4) = n$ for $n = 0 \pmod 3$.

\end{proof}

\section{Concluding Remarks and Open Problems}

Several open problems have arised during our work on this paper.  We list some of the more interesting ones:
\begin{itemize}
\item{The major role played in this paper by Lemma \ref{lemma_1} and its consequences suggest:

\emph{Problem 1}:  Find further structural conditions (along the lines indicated in Lemma \ref{lemma_1}) indicating that a graph $G$ is not saturated.}
\item{In Corollary \ref{cor_1}, we characterise saturated trees.  In a previous paper \cite{dmpclz} we characterised saturated graphs with $mp(G)=2$. This leads to the following:

\emph{Problem 2}:  Characterise $k$-saturated graphs for other families of graphs such as maximal outerplanar graphs, maximal planar graphs, regular graphs, etc.

\emph{Problem 3}:  Characterise saturated graphs with $mp(G)=3$.}
\item{The parameter $mp(G)$ can be very sensitive to edge-addition and edge-deletion, as shown in \cite{dmp2}.   Also Theorem \ref{thmcartprod} gives $h(n,7) \leq \frac{5n}{3}$ for $n = 0 \pmod 9$  while Theorem \ref{keven} gives $h(n,6)  \leq \frac{13n}{7} $ for $n = 0 \pmod 7$ .These facts  suggest the following monotonicity problem:

\emph{Problem 4}:  Is it true that, at least for $n$ large enough, depending on $k$, and for $k \geq 2$, $h(n,k+1) \geq h(n,k)$?

If true, this will have the immediate implication that the construction for $h(n,6)$ is not optimal and that in fact $h(n,6) \leq \frac{5n(1+o(1))}{3}$ by the above upper bound for $h(n,7)$.}

\item{The upper bound constructions given in Theorem \ref{thmcartprod} and Theorem \ref{keven} are probably not optimal.

\emph{Problem 5}:  Improve upon the upper bounds obtained in Theorems \ref{thmcartprod} and \ref{keven}}
\item{The lower bound given in Theorem \ref{lower} proved to be sharp in the case $k=4$.

\emph{Problem 6}:  Improve upon the lower bound $h(n,k) \geq n$ for $k \geq 5$.}
\item{In Proposition \ref{mod6} we have shown that  $h(n,5)  \leq \frac{7n}{6} + c( n \pmod 6)$.
 
\emph{Problem 7}: Determine $h(n,5)$  exactly . In particular, is it true that $h(n,5) = \frac{7n(1+o(1))}{6}$.}
\item{Lastly recall that $sat(n,k)  = n( 1 - c(k)) < n$ for every large  k and n.
 
\emph{Problem 8}:  Is it true that $h(n,k) \leq cn$ for some constant $c$ independent of $k$.}
\end{itemize}

\bibliographystyle{plain}

\bibliography{dmp3}
\end{document}